\documentclass[amsart,11pt]{article}

\usepackage{amsfonts}
\usepackage{amsmath}
\usepackage{amssymb}
\usepackage{amsthm}
\usepackage{bbm}
\usepackage{graphicx}
\usepackage{gensymb}
\usepackage{comment}
\usepackage{placeins}
\usepackage{geometry}
\usepackage{mathtools}
\usepackage{hyperref}
\usepackage{color}
\usepackage[normalem]{ulem}
\newcommand{\Z}{\mathbb Z}

\newcommand{\R}{\mathbb R}
\newcommand{\F}{\mathcal F}

\newtheorem{theorem}{Theorem}
\newtheorem{corollary}{Corollary}
\newtheorem{lemma}{Lemma}
\newtheorem{proposition}{Proposition}
\newtheorem{definition}{Definition}
\newtheorem{remark}{Remark}

\author{Toby Sanders and Rodrigo B. Platte \\ \small{School of Mathematical and Statistical Sciences, Arizona State University, Tempe, AZ, USA.} }
\date{}
\title{Multiscale Higher Order TV Operators for $\ell_1$ Regularization}

\begin{document}
\maketitle

\begin{abstract}
In the realm of signal and image denoising and reconstruction, $\ell_1$ regularization techniques have generated a great deal of attention with a multitude of variants. A key component for their success is that under certain assumptions, the solution of minimum $\ell_1$ norm is a good approximation to the solution of minimum $\ell_0$ norm. In this work, we demonstrate that this approximation can result in artifacts that are inconsistent with desired sparsity promoting $\ell_0$ properties, resulting in subpar results in {some} instances.  With this as our motivation, we develop a multiscale higher order total variation (MHOTV) approach, which we show is related to the use of multiscale Daubechies wavelets.  We also develop the tools necessary for MHOTV computations to be performed efficiently, via operator decomposition and alternatively converting the problem into Fourier space.  The relationship of higher order regularization methods with wavelets, which we believe has generally gone unrecognized, is shown to hold in several numerical results, although notable improvements are seen with our approach over both wavelets and classical HOTV.
\end{abstract}

\section{Introduction}  
Over the past couple of decades, $\ell_1$ regularization techniques such as total variation have become increasingly popular methods for image and signal denoising and reconstruction problems.  Along with TV \cite{ROF}, a large variety of approaches for similar $\ell_1$ regularization approaches have been proposed for an array of problems.  {Signal and image recovery methods continue to attract a great deal of interest due to the wide variety of potential applications and ever increasing means of various sensing mechanisms to acquire data.} To name a few, synthetic aperture radar {(SAR)} \cite{wei2010sparse,bhattacharya2007fast}, {magnetic resonance imaging (MRI) \cite{lustig2007sparse,1257394,4587391},} electron tomography\cite{Leary,SGP-ET}, and inpainting \cite{sanders2017subsampling,king2013image} are all image recovery {applications} that have advanced in part due to $\ell_1$ regularization methods, and in each case the approach can be tailored to the challenges that the particular application poses. With many problems such as MRI and electron tomography, the challenge is often to acquire as little data as necessary due to possible damage of the subject being imaged or because of time constraints, driving the need for inverse methods that can achieve the absolute best results from very limited and noisy data.

The mathematical description of the general problem we are interested in is to recover a signal or image $f\in \R^N$, from noisy measurements $b$ of the form $b = Af+\epsilon$, where $A\in \R^{m \times N}$ is some sensing matrix that approximates the physical model of the particular problem.  Then the $\ell_1$ regularized solution is given by
\begin{equation}\label{gen-l1}
f_{rec} = \arg \min_f \Big\{ \| Af - b \|_2^2 + \lambda \| T f \|_1 \Big\} ,
\end{equation}
where $T$ is some sparsifying linear transform and $\lambda$ is a parameter that balances the effects of the data and regularization terms.  The appropriateness of this approach is that some prior knowledge of the signal suggests that $Tf$ is sparse, and that the formulation with the $\ell_1$ norm encourages such sparsity \cite{eldar2012compressed,CSincoherence,candes2006robust}.  In many applications, some knowledge of the appropriate transform is available, particularly with images and for other signals, this knowledge is in the form of some ``smoothness.''  

In the case of TV, the sparsifying transform is given by $T : \R^N \rightarrow \R^{N-1}$, where $(Tf)_i = f_{i+1} - f_i$.  The general idea for this approach is that the signal $f$ is assumed to be piecewise constant with a few discontinuities, in which case $Tf$ is sparse.  If this is not precisely true, this approach still effectively reduces unwanted oscillations at the cost of the well documented stair-casing effect \cite{HOTV,blomgren1997total}.  However, for more {general piecewise smooth functions} higher order TV (HOTV) regularization methods are effective \cite{HOTV,TGV,hu2012higher}, and they do not suffer from the stair-casing effects.  In this case the transform maps $f$ to approximations of discrete derivatives of $f$, e.g. higher order finite differences of $f$.

Another popular choice for $T$ are wavelet transforms \cite{starck2010sparse,mallat2008wavelet,lustig2007sparse}.  For instance, such a transform can be written as $T : \R^N \rightarrow \R^{N}$, where $(Tf)_j = \langle f , \psi_j \rangle$ and $\psi_j$  are orthonormal so that $f = \sum_j \langle f , \psi_j \rangle \psi_j$.  The idea here is that for appropriately smooth signals, most of the signal's energy is captured in the few low frequency, larger scaled elements of the basis.  Thus most of the coefficients can be neglected, and thus a sparse approximation of $f$ exists with respect to the basis.

\subsection{Discussion and Contribution}

The crux of general $\ell_1$ regularization methods is that recovering a signal with the most sparse representation, that is recovering the solution with the smallest so called $\ell_0$ norm, is often equivalent to its convexly relaxed variant of recovering the signal with the smallest $\ell_1$ norm, which is a field of study called compressed sensing (CS) \cite{eldar2012compressed,CSincoherence,candes2006robust}.  Although convex $\ell_1$ optimization algorithms are useful in promoting sparsity, some small nonzero coefficients may still persist, an obvious sign that the assumptions needed for the exactness guarantees given by CS theory sometimes do not hold in practice.  This observation is largely the original motivation our present work in developing a multiscale HOTV approach related to multiscale wavelet regularization.

Much work has been devoted to understanding and developing sparsity promoting regularization methods, which are related to our current work.  Numerous variants of higher order TV methods have been proposed \cite{HOTV,Archibald2015,hu2012higher}.  For example, in \cite{Archibald2015} the authors propose an edge detection operator that annihilates polynomials, which leads them to operators close to finite difference matrices. In \cite{HOTV} a combination of a TV regularizer with a quadratic second order regularizer is developed in the continuous domain to eliminate staircasing effects.  Likewise, several authors have shown that using some combination of first and second order methods to be beneficial \cite{VOTV,TGV,setzer2011infimal,chambolle1997image}.  Unfortunately, since there are multiple regularization terms these methods typically introduce additional parameters that need to be tuned.  In terms of theory, it has been well documented that under certain conditions TV and HOTV are equivalent to reconstruction with splines \cite{unser2017splines,steidl2006splines}, i.e. the solution of such methods recovers a piecewise polynomial with a sparse set of jumps.  

TV denoising in particular has several very interesting equivalences.  It is well known that TV denoising and other more general first order denoising methods are equivalent to smoothing with a certain nonlinear diffusion models\cite{scherzer2000relations}, a typical result of writing the equivalent Euler-Lagrange equations.  Perhaps discussed less frequently and most related to the observations in our current work, TV denoising is equivalent to soft threshold denoising with the highest frequency basis elements of the Haar wavelets \cite{steidl2002relations,steidl2004equivalence}, in particular with the so called cycle spinning \cite{kamilov2012wavelet}.  In general however, the main difference between these methods is that with TV the smoothness analysis is limited to the finest scales, whereas wavelet regularizations promote function smoothness at multiple scales.  A main contribution of this article is to expand further on the relationship between wavelets in $\ell_1$ regularization and those $\ell_1$ methods related to HOTV.  In regards to extension of wavelets, a number of multidimensional generalizations have been invented including curvelets and shearlets  \cite{guo2007optimally,kutyniok2012shearlets,starck2010sparse}, which are primarily used for sparse function approximation and improve the approximation rates in two and three-dimensions compared with their one-dimensional counterparts.


The method we develop here an alternative for HOTV regularization which we refer to as multiscale HOTV (MHOTV).  In contrast to previous work, our approach considers combining both a multiscale approach and higher order TV methods for the class of image reconstruction problems.  The motivation for such an approach is in observable sub par results due to the relaxation of the sparsity promotion through the $\ell_1$ norm, contrary to the aforementioned results with splines \cite{unser2017splines,steidl2006splines}.  In light of this, we determined this calls for analysis of the function behavior at multiple scales.  As can be deduced, this multiscale {strategy} is similar to the treatment of wavelets, and we argue that our approach is indeed related to the use of Daubechies wavelets, with the main divergence coming in the orthogonality and/or frame conditions prescribed by the wavelets.  {Orthogonality} may be unnecessary for general $\ell_1$ regularization techniques, although fundamental to thresholding denoising techniques, and the relaxation of this condition in our approach allows for better localization of the {transform}.  In the development of MHOTV, we carefully address the computational concerns associated with our approach through the use of both the FFT and operator decompositions.  We are able to show through several numerical examples that MHOTV provides a notable improvement to the current alternatives.

The organization of the remainder of the article is as follows.  In section \ref{sec3} we define the HOTV operators and the corresponding multiscale generalizations.  We also motivate the approach via a numerical example, and make the connection with Daubechies wavelets.  In section \ref{sec4} we precisely define the MHOTV $\ell_1$ regularization model and give precise normalizations to deal with proper parameter selection.  In section \ref{sec5} we address the computational concerns associated with calculating MHOTV coefficients, devising two distinct ways that they can be calculated in an efficient manner.  In section \ref{sec6} we provide numerical results for 1-D and 2-D reconstruction problems, showing that MHOTV is an improvement to the original HOTV and the related Daubechies wavelets.  Some proofs and definitions are provided in the appendix.

\section{HOTV and Multiscale Generalizations}\label{sec3}
As an alternative to TV regularization, {general} order TV methods have been shown to be effective for $\ell_1$ regularization \cite{HOTV,TGV,SGP-ET,Archibald2015}.  The TV transform can be thought of as a finite difference approximation of the first derivative, thus annihilating a function in locations where the function is a constant, i.e. a polynomial of degree zero.  Likewise, higher order TV transforms can be considered higher order finite difference approximations to higher order derivatives, thus annihilating {{the corresponding} degree polynomials.  With this in mind, we have the following definition:

\begin{definition}[Finite Differences]
Let $\phi_k \in \R^N$ be defined by
\begin{equation}
(\phi_k)_m = \begin{cases}
(-1)^{k} & \mbox{if } \, m=0\\
0 & \mbox{if  } \, 1\le m<N-k \\
(-1)^{k+m+N} {k \choose N-m} & \mbox{if }\, N-k \le m < N
\end{cases}.  
\end{equation}
Then for $f\in\R^N$, the periodic $k^{th}$ order finite difference of $f$ is given by 
$$
f*\phi_k,
$$
where $*$ denotes the discrete convolution.
\end{definition}

\begin{remark}
The convolution in this definition (and in general) can be represented by multiplication with a circulant matrix $\Phi_k$, where each row of $\Phi_k$ holds a shifted version of $\phi_k$.  An example of the matrix in the case $k=2$ is given in (\ref{Phi_2}).  Note that this definition uses a periodic extension of $f$ and can be ignored by dropping the last $k$ rows of $\Phi_k$.
\end{remark}
\begin{equation}\label{Phi_2}
\Phi_2 = \left( \begin{array}{c c c c c c}
               1 & -2 & 1 & 0 &  \dots & 0\\
               0 & 1 & -2 & 1 & \dots & 0 \\
               0 & 0 & 1 & -2 & \dots & 0 \\
               \vdots &   & \ddots  &  & \dots & \vdots \\
               1 & 0 & \dots & & 1 & -2\\
               -2 & 1 & \dots & & 0 & 1
              \end{array}\right).
\end{equation}

With this definition, the HOTV model can be said to recover
\begin{equation}\label{HOTV-model}
 f_{rec} = \arg \min_f \Big\{\|Af - b\|_2^2 + \lambda \|\Phi_k f\|_1 \Big\}.
\end{equation}

Unfortunately for many real world imaging problems the equivalence between $\ell_1$ and $\ell_0$ may not hold in practice, yet the $\ell_1$ regularization still tends to encourage favorable solutions.  In terms of the sparsity promoting transform, this means that the transform of the recovered function may not be truly sparse, but most of the values are instead relatively close to zero.  For HOTV, this means that a local Taylor expansion of the recovered function will still contain some small nonzero higher order coefficients, yet essentially unobservable at the very local scale.  In other words, at some point $t$, there exists a polynomial expansion of minimal degree of $f$ given by 
\begin{equation}\label{taylor}
 f(x) {\approx} \sum_{m=0}^M \alpha_m(t) \frac{(x-t)^m}{m!} ,
\end{equation}
which holds for all $x$ within some small interval $I$ around the point $t$.  Ideally a solution given by the order $k$ HOTV model recovers a solution so that the coefficients $\alpha_m(t)$ vanish for $m\ge k$.  The $\ell_1$ model allows for these coefficients to remain, although very small, and the function still \emph{appears} to essentially be a polynomial of degree less than $k$.
However, when this behavior persists over many points at a larger scale, the result can be a function which looks more like a trigonometric polynomial rather than an algebraic one.  

This phenomenon is demonstrated in Figure \ref{fig1}, where a piecewise polynomial of degree two was reconstructed from random noisy samples with 50\% sampling\footnote{The number of samples is half the number of grid points.} using TV and HOTV regularizations.  The sampling matrix $A\in \R^{N/2 \times N}$ is constructed so that a random 10\% of its entries are set to be nonzero, where these nonzero values are uniformly distributed over $[0,1]$.  The samples were corrupted with normaly distributed mean zero noise.  Two different grid sizes are demonstrated, 256 and 1024, and it can be observed that these small oscillations become increasingly abundant with more grid points.  However, in the bottom of the figure, the third order finite difference of the HOTV3 solution plotted in logarithmic scale shows that locally this oscillatory behavior {results in almost exact low order polynomials, although \emph{very} small amplitudes persist in the transformed domain and} thus not truly sparse in the $\ell_0$ sense.  Nevertheless, all regularization approaches should still be deemed useful, as evidenced by the least squares solution shown.

\begin{figure}
 \centering
 \includegraphics[width=0.5\textwidth]{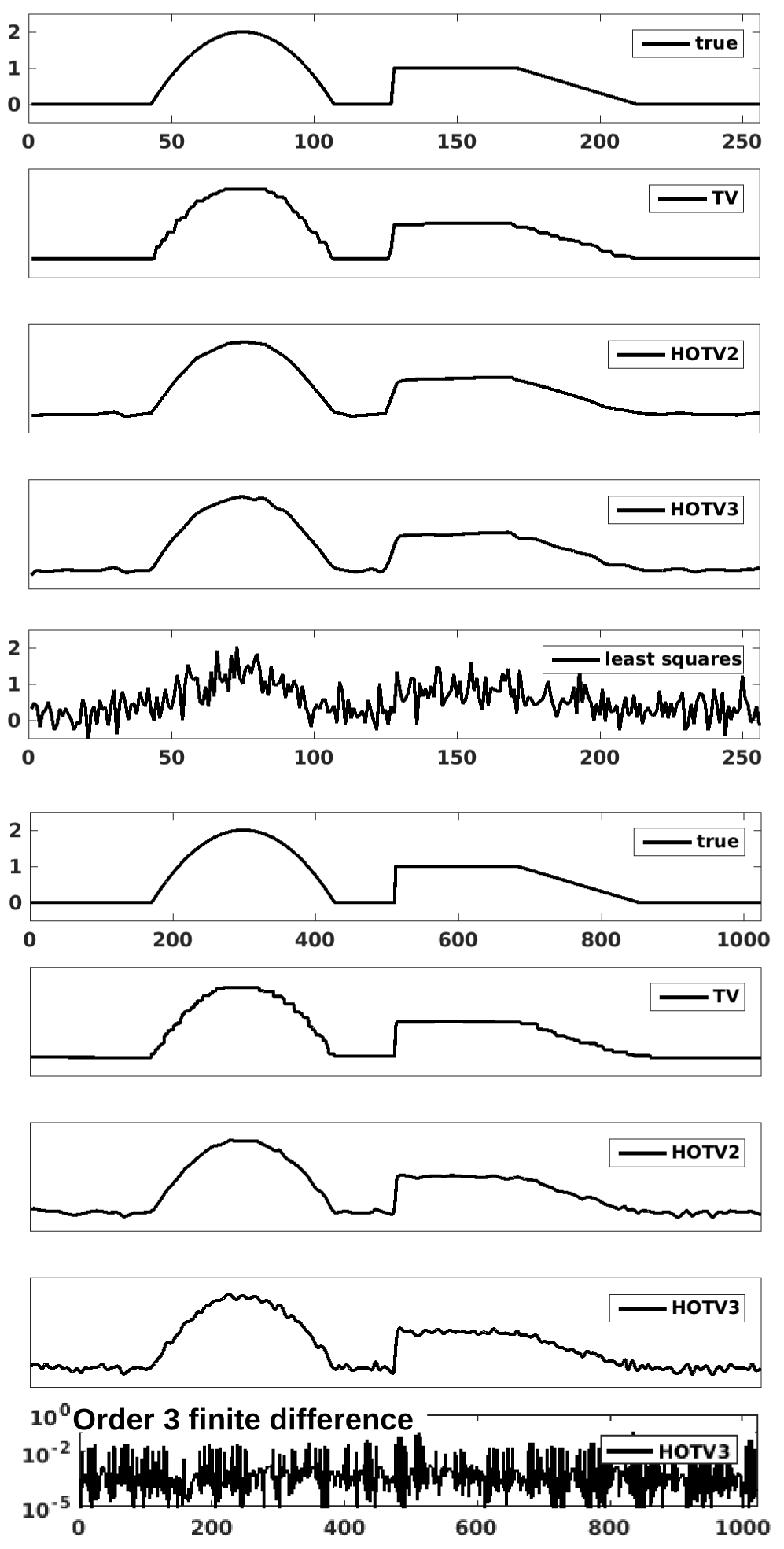}
 \caption{Rows 2-4 and 7-9 reconstruction of a piecewise polynomial function of degree two shown in the top row over 256 (top 5 plots) and 1024 (bottom 5 plots) {points} from random sampling at 50\%.  {The corresponding least squares solution is shown in the fifth plot, and the 3rd order finite difference of the HOTV3 solution over the 1024 grid is shown on the bottom.}}
 \label{fig1}
\end{figure}

Due to this phenomena we propose a multiscale HOTV approach, which considers the regularization transform at multiple scales.  The idea is that a larger stencil would penalize these oscillations even with the $\ell_1$ norm.  As TV generalizes to the Haar wavelet by stretching and scaling of the elements, we propose the same with HOTV.  {To this end we} give the following definition.

\begin{definition}[Multiscale Finite Differences]
Let $\phi_{k,j} \in \R^N$ be defined by 
\begin{align}
& (\phi_{k,j})_m =\label{MHOFD-vect} \\ 
& \begin{cases}
(-1)^{k} & \mbox{if } \, m=0\\
0 & \mbox{if }  \, 1 \le m \le  N-j(k+1)\\
(-1)^{k+\lfloor \frac{N-m}{j} \rfloor} {k \choose \lfloor \frac{N-m}{j} \rfloor}
& \mbox{if } \, N-j(k+1) < m < N 
\end{cases}. \nonumber
\end{align}
Then for $f\in \R^N$, the periodic $k^{th}$ order finite difference of scale $j$ of $f$ is given by 
$$
f*\phi_{k,j},
$$
where $*$ denotes the discrete convolution.
\end{definition}

\begin{remark}
Again, this convolution can be represented as multiplication with a circulant matrix $\Phi_{k,j}$.  An example of $\Phi_{k,j}$ in the case $k=2$ and $j=2$ is given in (\ref{Phi_22}).
\end{remark}
\begin{equation}\label{Phi_22}
\Phi_{2,2} = \left( \begin{array}{c c c c c c c c c}
               1 & 1 & -2 & -2 & 1 & 1 & 0 & \dots & 0\\
               0 & 1 & 1 & -2 & -2 & 1 & 1 & \dots & 0 \\
               0 & 0 & 1 & 1 & -2 & -2 & 1 & \dots & 0 \\
               \vdots &   &  & \ddots &  && &\dots &  \vdots \\
               -2 & -2 & 1 & 1 & 0& 0 &  \dots & 1 & 1\\
               1 & -2 &-2 & 1 & 1 & 0 & \dots & 0 & 1
              \end{array}\right).
\end{equation}

\subsection{MHOTV Reconstruction Model}\label{sec4}
We now present the general model for MHOTV reconstruction.  Generally speaking, we still use the model presented in (\ref{gen-l1}), where $A$ maps the unknown function $f$ to some perhaps noisy measurements given by $b$, from which we use to reconstruct $f$.  Our sparsity promoting transforms are now given by the matrices $\Phi_{k,2^j}$, for $j=0,1,\dots, \ell$, where $\ell$ is the maximum scaling of the operator used and $k$ is the chosen order.  Setting our maximum scaling to $\ell=0$ corresponds to the traditional HOTV approach.  Although not completely necessary, we choose a dyadic scaling of the operators, similar to the treatment of wavelets.  As with wavelets, we will show that this is convenient for computational purposes.  Finally then our reconstruction model is given by
\begin{equation}\label{MHOTV-model}
 f_{rec} = \arg \min_f \Big\{ \| A f - b\|_2^2 + \frac{\lambda}{\ell+1} \sum_{j=0}^\ell 2^{-(j+k-1)} \| \Phi_{k,2^j} f \|_1\Big\}.
\end{equation}
The factor of $2^{-j}$ is a normalization that accounts for the increasing norms of each operator, which would otherwise weigh too heavily to the largest scaling operator \footnote{This is akin to the dyadic scaling of the wavelet basis elements after the dyadic stretching.}.  The scaling of the parameter $\lambda$ by $\ell+1$ simplifies the selection of the parameter, which would otherwise need to be manually scaled by such a factor to account for the number of scales being used. By similar reasoning, the additional scaling by $2^{-k+1}$ is used to account for the order $k$ of the method \cite{sanders2018parameter}.

\subsection{Relationship to Daubechies Wavelets}

Wavelets can be characterized as an orthonormal basis that is generated through a multiresolutional analysis \cite{daubechies1992ten,mallat2008wavelet}.  The multiresolutional analysis leads to the shifting and dyadic stretching and scaling of a single generating mother wavelet, analogous to our treatment of MHOTV by shifting and stretching of a single row or element of the matrices $\Phi_k$.
From this very general characterization, there are a number of parameters in the design of the wavelets.  {For Daubechies wavelets} the smoothness is characterized by the number of vanishing moments, i.e. the number of polynomial orders to which the wavelet is orthogonal.  A wavelet with $k$ vanishing moments acts as a multiscale differential operator of order $k$.  As a trade off, an increasing number of vanishing moments chosen for the wavelet basis results in an increase in the support of the wavelet functions, and Daubechies wavelets are designed to yield the orthonormal wavelet basis of minimal support given a selected number of vanishing moments \cite{mallat2008wavelet}.

To develop a basic mathematical description of a wavelet expansion, suppose we want to represent a \emph{uniform pixelated} function with $2^n$ pixels on $[0,1]$ in terms of the wavelet basis.
Then denoting our scaling function and mother wavelet with $k$ vanishing moments by $\varphi_k$ and $\psi_k$ respectively, we have the following orthonormal wavelet representation
\begin{equation}\label{wave-eq}
f = \sum_{t=0}^{2^{\ell} - 1} \langle f , \varphi_{k,\ell,t} \rangle \varphi_{k,\ell,t} 
	+ \sum_{j=\ell}^{n-1} \sum_{t=0}^{2^{j}-1} \langle f , \psi_{k,j,t} \rangle \psi_{k,j,t} .
\end{equation}
Here, 
$
\psi_{k,j,t}(x) = 2^{j/2} \psi_k \left(2^j x - t \right)$
and similarly for $\varphi_{k,j,t}$,
i.e. shrinking and scaling of the of the generating wavelet functions.  The parameter $\ell$ is a positive integer with $0\le \ell \le n$, and the value
$n-\ell$ is said to be the number of \emph{scales} in the wavelet expansion \footnote{For $\ell=n$ it is understood that the second sum is removed.}.
With the representation in (\ref{wave-eq}), the coefficients for the scaling functions in the first sum capture most of the energy of the signal, and the wavelet coefficients $c_{k,j,t} = \langle f,\psi_{k,j,t} \rangle$ \emph{vanish} for values of $t$ where $f$ is a polynomial of degree $k-1$ over the support of $\psi_{k,j,t}$.  {For $\ell_1$ regularization, we only need to be concerned with regularization of the wavelet coefficients in (\ref{wave-eq}), and thus the coefficients for the scaling functions in the first sum are not included in the regularization.  
}

\begin{figure}
\includegraphics[width=.5\textwidth]{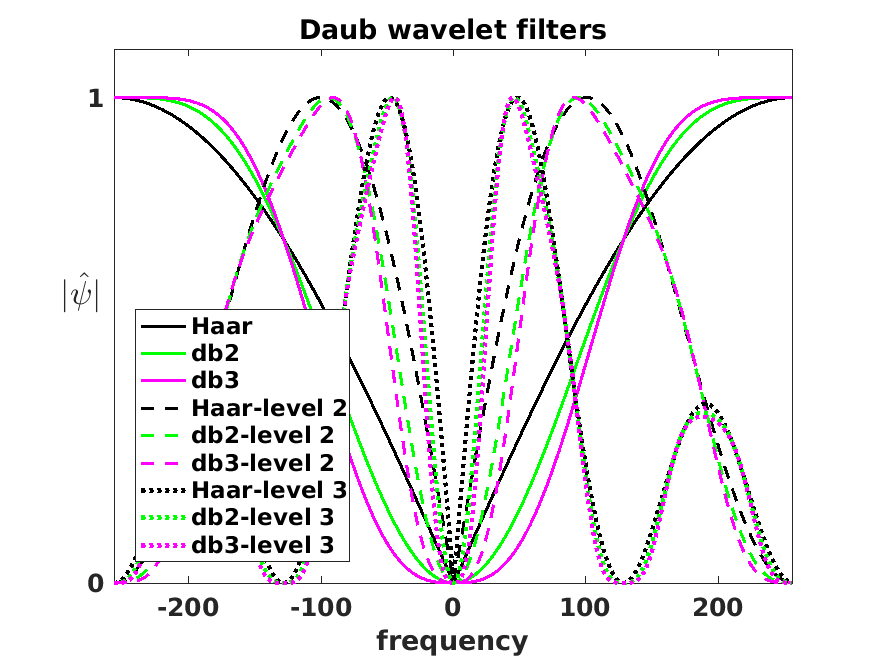}
\includegraphics[width=.5\textwidth]{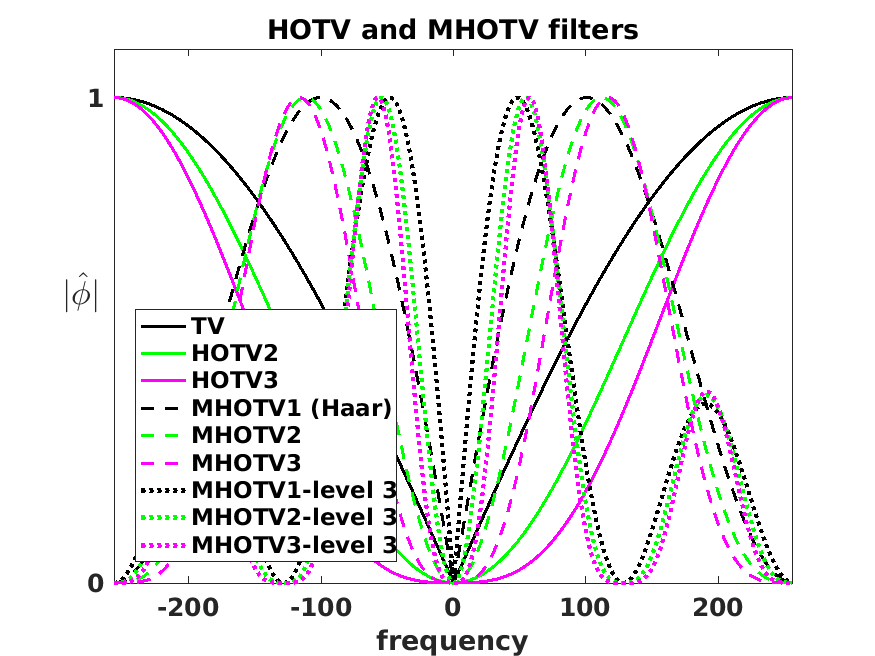}
\caption{The filters in Fourier space of wavelet and MHOTV convolution functions.}
\label{fig: filters}
\end{figure}

Connecting  these ideas to HOTV, we see that these {transforms} are playing similar roles.  Both are prescribed by the number of vanishing moments, or in the language of HOTV, the highest order polynomial that is annihilated by the approach.  Furthermore, both are designed to yield minimal support given the number of vanishing moments.  The crucial difference lies in the orthogonality condition prescribed by wavelets, which further increases the support of the wavelet elements.  We emphasize again, that this condition is fundamental to compression and threshold denoising methods, but not necessarily useful with general image reconstruction problems.

Finally, one additional technique utilized for $\ell_1$ regularization and denoising as well is the use a wavelet frame by taking all possible shifts for each scaling of the wavelets, which is sometimes referred to as translational invariant cycle spinning \cite{spinning,temizel2005wavelet,kamilov2012wavelet}.  This eliminates the lack of translation invariance of a wavelet basis that can otherwise result in unwanted artifacts near discontinuities.  With this in mind, we may define the wavelet frame elements by
$$
\tilde \psi_{k,j,t}(x)  = 2^{j/2} \psi_k \left( 2^j (x - t2^{-n})\right), \quad t = 0 , 1 , \dots, 2^{n-1}.
$$
Then the \emph{averaged} wavelet frame representation of $f$ may be written as
\begin{align*}\label{wave-eq-frame}
f & = \sum_{t=0}^{2^{\ell} - 1} \langle f , \varphi_{k,\ell,t} \rangle \varphi_{k,\ell,t} 
	+ \sum_{j=\ell}^{n-1}2^{j-n} \sum_{t=0}^{2^{n}-1} \langle f , \tilde \psi_{k,j,t} \rangle \tilde \psi_{k,j,t} \\
	& = \sum_{t=0}^{2^{\ell} - 1} \langle f , \varphi_{k,\ell,t} \rangle \varphi_{k,\ell,t} 
	+ \sum_{j= \ell}^{n-1} 2^{j-n} \Psi_{k,j}^T (f * \psi_{k,j,0}(-x)),
\end{align*}
where $\Psi_{k,j}^T = (\tilde \psi_{k,j,0}, \tilde \psi_{k,j,1}, \dots , \tilde \psi_{k,j,2^n - 1})$.  Hence a wavelet approach promotes sparsity with respect to the vectors $f*\psi_{k,j,0}$, or equivalently with respect to $\Psi_{k,j}f $.  Then a regularization norm in this setting takes the form
\begin{equation}\label{wave-reg-norm}
\sum_{j=\ell}^{n-1} \| \Psi_{k,j} f \|_1,
\end{equation}
which is analogous to our regularization norm in (\ref{MHOTV-model}).  For wavelets, the scalings are inherent to function definitions, and the dyadic stretching of the elements is indicated by $j$ as opposed to $2^j$.

The case when $\ell=n-1$ would be most closely related to the original HOTV, and for smaller values of $\ell$ the wavelets are more comparable to the MHOTV development in this article.} 

Since computing both MHOTV operators and wavelets coefficients are convolutional operations, we may visualize their corresponding filters in Fourier space, providing us another basis for comparison, which we have done in Figure \ref{fig: filters}.  Each of these can be interpreted as high pass filters, where the higher levels pass increasingly lower frequencies.  A very close similarity of the wavelet filters and MHOTV filters can be observed in Figure \ref{fig: filters}, providing a strong visual confirmation to our preceding discussion on the close relationship between the two.

\section{Fast Calculation of MHOTV Operators} \label{sec5}
Calculation of traditional HOTV coefficients is a computationally inexpensive task, due to the sparsity of the matrix operator.  However, with increasing dyadic scales the direct calculation increases exponentially.  Due to this, in the proceeding section we develop two distinct approaches that show that these calculations can be carried out with linear increase in the flop count with respect to the number of scales used.

Fast computation of standard HOTV can be done in several ways.  One can construct the sparse matrix $\Phi_k$ and perform matrix computations directly, a calculation with runtime of $kN$ flops.  One could make use of other {procedures}, such as MATLAB's ``diff" command which  {requires the same} flop {count} without storing the matrix.  With MHOTV, these approaches become less appealing.  With matrix construction, if one is using several scales, then several matrices need to be computed and stored, and the matrices become significantly less sparse for larger scales.  The``diff" command cannot be implemented directly for larger scale HOTV operators.  

Another alternative is to use the Fourier convolution theorem to perform the convolution {operation} via a product in Fourier space.  For the traditional HOTV operators, this can be fairly slow compared with the matrix and ``diff" approach, since the necessary two discrete Fourier transforms would require $\sim 2 N\log_2 N$ flops compared with the $kN$ flops for the alternative {implementations}.  However, this {method} is relatively comparable for MHOTV, since the Fourier transforms only need to be computed once to determine the coefficients at all scales.

We outline two {procedures} for efficient calculation of MHOTV.  First, we describe the Fourier approach, where we derive precise formulas for the MHOTV Fourier filters.  Second, we describe an alternative efficient approach by decomposition of the MHOTV matrix operators.

\subsection{Computation via Fourier Transforms}
By the Fourier convolution theorem, the MHOTV operators can be computed {as multiplications} in Fourier space, i.e.
\begin{equation}
 f*\phi_{k,j} = F^{-1} \left( F(f)\cdot F(\phi_{k,j})\right) ,
\end{equation}
where $F$ denotes the discrete Fourier transform.  Although this can be numerically computed, it is a convenient to have an exact formula for the discrete Fourier transform of $\phi_{k,j}$.  Moreover, analytic determination of $F(\phi_{k,j})$ allows us to generalize the MHOTV to fractional orders.
\begin{proposition}
The DFT of the vector $\phi_{k,j}$ defined in (\ref{MHOFD-vect}) has an explicit expression given by
\begin{equation}\label{eq: DFTeq}
 F( \phi_{k,j} )_\xi = \frac{\left( e^{ \frac{i2\pi\xi j}{N}} - 1\right)^{k+1}}{e^{\frac{i2\pi\xi}{N}}-1},
\end{equation}
for $\xi = 0, 1, \dots, N-1$.
\end{proposition}
\begin{proof}
The expression for the $\xi^{th}$ Fourier coefficient in the DFT of $\phi_{k,j}$ is given by
\begin{equation}\label{DFT}
F( \phi_{k,j} )_\xi = \sum_{m=0}^{N-1} (\phi_{k,j})_m e^{\frac{-i2\pi\xi}{N}m}.
\end{equation}
Notice that the terms $1\le m \le N - j(k+1)$ vanish by definition of $\phi_{k,j}$.  For the latter terms, we make the substitution $n=N-m$ and flip the sum to give the expression
\begin{equation}
 F(\phi_{k,j})_\xi = \sum_{n=0}^{j(k+1)-1} (-1)^{k + \lfloor{\frac{n}{j}}\rfloor} {k \choose \lfloor \frac{n}{j} \rfloor } e^{\frac{-i2\pi\xi}{N}(N-n)},
\end{equation}
where the term $n=0$ corresponds to $j=0$ and the following indices $n=1,2,\dots, m(k+1)-1$, correpsond to $j=N-1,N-2, \dots, N-m(k+1)+1$, respectively.
Notice that we may drop the $N$ in the numerator of the exponential and that the values of $\phi_{k,j}$ repeat over strings of length $j$.  Therefore each of these corresponding strings of exponential terms in (\ref{DFT}) get the same weights, leading to the following sum:
\begin{equation}
 F( \phi_{k,j} )_\xi
   = \sum_{m=0}^k \left( (-1)^{m+k}{k\choose m} \left[ \sum_{\ell=0}^{j-1} e^{\frac{i2\pi\xi}{N}(jm+\ell)} \right] \right) .
\end{equation}
Here the inner sum represents the $j$ consecutive terms in (\ref{DFT}) that receive the same weights from $\phi_{k,j}$, namely $(-1)^{m+k}{k\choose m}$.  Switching the order of summation, we recognize the sum over $m$ as a binomial expansion leading to
\begin{align*}
 F( \phi_{k,j} )_\xi  &= \sum_{\ell=0}^{j-1} \sum_{m=0}^k (-1)^{m+k}{k\choose m} e^{\frac{i2\pi\xi}{N}(jm+\ell)}\\
 & =  \sum_{\ell=0}^{j-1} \left( e^{\frac{i2\pi\xi}{N}j} - 1 \right)^k e^{\frac{i2\pi\xi}{N}\ell}.
\end{align*}
The remainder of the proof follows by elementary calculations.
\end{proof}

\subsection{Fast Computation via Operator Decomposition}
In this section, we give a decomposition for the matrix operator $\Phi_{k,2^j}$ and describe how this decomposition can be used for rapid calculation of MHOTV operators.  The decomposition of $\Phi_{k,2^j}$ is given in the following theorem.

\begin{theorem}\label{thm1}
 Let the matrix $P_j$ with entries $\{p_{m,n}\}_{m,n=1}^N $ be defined by 
 \begin{equation}\label{pvalues}
 p_{m,n} = \begin{cases}
            1 & \mbox{if } m=n\\
            1 & \mbox{if } n= (m+j-1)\bmod{(N)} + 1 \\
            0 & \mbox{if } otherwise
           \end{cases}.
 \end{equation}
Then the following holds:
\begin{enumerate}
\item The entries of $P_j^{k+1}$, which we denote by $\{p_{m,n}(j,k)\}_{m,n=1}^N$, are given by
$$
p_{m,n}(j,k) = \begin{cases}
              {k+1 \choose \ell} &\mbox{if } n = (m+j\ell-1)\bmod({N)}+1\\
              0 &\mbox{if } otherwise
             \end{cases},
$$
where it is implied $\ell$ is an integer satisfying $0\le \ell \le k+1$.
\item  $\Phi_{k,2^j}$ has the decomposition 
\begin{equation}\label{decomp2}
 \Phi_{k,2^j} = P_j^{k+1} P_{j-1}^{k+1} \cdots P_1^{k+1} \Phi_k 
\end{equation}
and therefore
\begin{equation}\label{decomp3}
 \Phi_{k,2^j} = P_j^{k+1} \Phi_{k,2^{j-1}}.
\end{equation}

\item The equality in (\ref{decomp2}) holds for any rearrangement of the product of matrices.
\end{enumerate}
\end{theorem}
{The proof of this theorem is given in the appendix. The matrices $P_2$ and $P_2^2$ are shown below to illustrate the sparse structure of these operators:}
\begin{align*}\label{pmatrices}
 P_2 &= \left( \begin{array}{cccccc}
        1 & 0 & 1 & 0 & \dots & 0\\
        0 & 1 & 0 & 1 & \dots & 0\\
        \vdots & & \ddots & & \dots &  \vdots\\
        0 & 1 & 0 & \dots & & 1
       \end{array}\right) , \\
P_2^2 &= \left( \begin{array}{ccccccc}
                1 & 0 & 2 & 0 & 1 & \dots & 0\\
                0 & 1 & 0 & 2 & 0 & \dots & 0\\
                \vdots & & \ddots & & & \dots &  \vdots\\
                0 & 2 & 0 & 1 & 0 & \dots & 1
               \end{array}\right).       
\end{align*}

\begin{proposition}\label{prop2}
 Direct calculation of $\Phi_{k,2^j}$ requires $2^j Nk$ flops.  The same calculation using the decomposition in (\ref{decomp2}) requires $j N(k+1) + Nk$ flops.  The same calculation using the Fourier method requires $ 2Nlog_2 N + N$.
\end{proposition}
Proposition 2 is a direct result of Theorem \ref{thm1}, the Fourier convolution theorem combined with the FFT, and the flops required for the direct calculation.  We assume that the FFT and inverse FFT can be computed in $N\log_2 N$ flops, although the exact count is somewhat vague, depending on the precise algorithm and if $N$ is a power of 2.  To compute the full set of operators, we can get away with less flops then adding the flops for each level.  If we use the decomposition approach to calculate the operators as determined by (\ref{decomp2}), the associated computations are limited to that at the highest scale, since the intermediate scales are determined in this calculation as pointed out in (\ref{decomp3}).  If we use the Fourier approach for calculating the coefficients in (\ref{MHOTV-model}), only one foward FFT is required for the function $f$.  Then the product of $F(f)$ and $F(\phi_{k,2^j})$ must be computed for each $j$, as well as the inverse FFT for each of these products. The observations lead to the following corollary.

\begin{corollary}
 Let $T$ be the matrix containing the complete set of $\ell +1$ operators involved in the MHOTV $\ell_1$ regularization norm, so that $T^T= [\Phi_{k,1}^T , \Phi_{k,2}^T , \dots , \Phi_{k,2^\ell}^T].$  Then calculating $T$ using the operator decomposition given in Theorem 1 requires $\ell N(k+1) + Nk$ flops.  Calculating $T$ using the Fourier approach requires a total flop count of $(\ell +2) N\log_2 N+(\ell+1) N$.
\end{corollary}

A few concluding remarks are in order.

\begin{remark}
 All of the results presented are for 1-D signals.  For higher dimensions say 2-D, the operators can be applied along each row and column, and the flop count is only doubled, disregarding the likely increased number of indices.
\end{remark}

\begin{remark}
 To solve (\ref{MHOTV-model}), we use the well establised alternating direction method of multipliers (ADMM) \cite{Li2013,bregman,Zhang}.  This approach introduces splitting variables that allows one to split the objective functional into equivalent subproblems that can be solved relatively fast.  Our algorithm can be downloaded at \cite{toby-web}, and some of the simulations in the proceeding section can also be found there.
\end{remark}

\section{Numerical Experiments}\label{sec6}
\subsection{Repeat of 1-D Simulations}\label{sec: repeat1d}
\begin{figure}
 \centering
 \includegraphics[width=.5\textwidth]{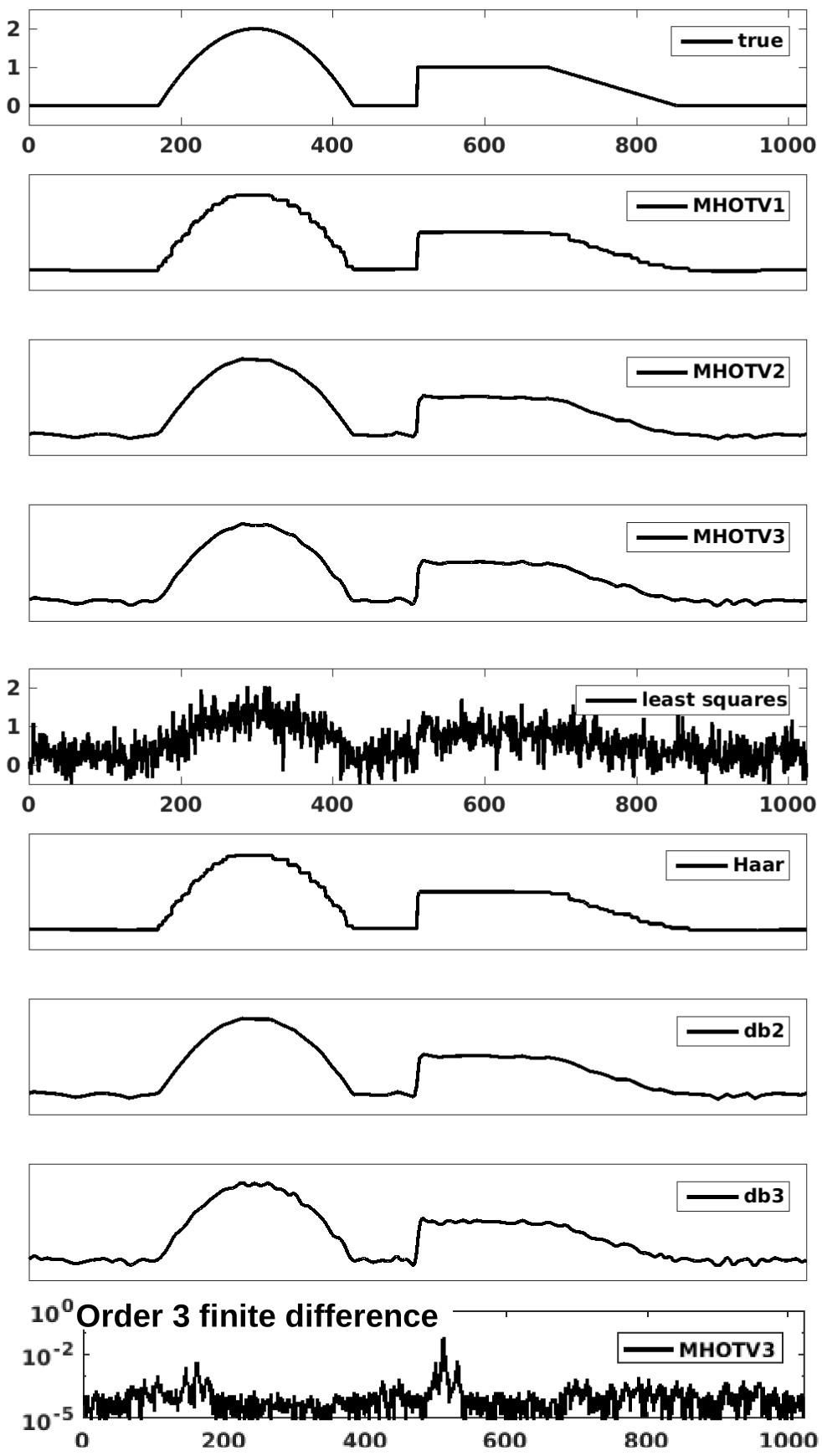}
 \caption{Reconstruction of a piecewise polynomial function of degree two over 1024 stencil from random sampling at 50\%.  Three scales are used for both the Daubechies wavelets and multiscale HOTV.}
 \label{fig2}
\end{figure}
{To compare MHOTV and wavelet regularized reconstructions we repeat the numerical examples presented in Figure \ref{fig1} with the same noisy data used for the HOTV reconstruction. The corresponding reconstruction with MHOTV and wavelets are presented in Figure~\ref{fig2}.} Recall that the measurements were collected at a 50\% sampling rate and corrupted with normally distributed mean zero noise.  For the multiscale HOTV and wavelets, three scaling levels were used.  The selection of the regularization parameter $\lambda$ was set to the same value for each order for HOTV and the wavelets, where we used a similar normalization approach for the wavelets coefficients as presented in (\ref{MHOTV-model}).


{The results in Figure \ref{fig2} were generated with orders 1, 2, and 3.} The order is indicated with the numbers next to the approach in the legends, e.g. we denote the order $k=3$ MHOTV approach with MHOTV3. For a baseline comparison, the least squares solution is shown as well.  Compared with the corresponding $1024$ reconstructions from HOTV in Figure \ref{fig1}, these solutions show clear improvements, particularly with the higher orders.  As we expect, although the MHOTV1 and Haar wavelet coefficients are computed in a different manner, the resulting reconstruction are identical since the models are theoretically equivalent.  They both exhibit the staircasing and noise effects in precisely the same locations.  The higher order approaches also show many similar effects of the noisy features, exhibiting certain oscillatory features with the same general behavior in precisely the same locations.  However, with the higher orders, these approaches are not equivalent and MHOTV provides regulatory information at finer scales due to the minimal support of the transform elements.  The result appears to be a modest improvement in the resulting reconstructions. 

Finally, in the bottom of the figure the third order finite difference of the MHOTV3 solution is plotted in logarithmic scale.  Comparing this with the original HOTV3 finite difference in Figure \ref{fig1}, we observe that the solution exhibits much better sparsity with respect to this transform domain, as desired.

\subsection{2-D Tomographic Simulations}
In this section we investigate the regularization methods on the common 2-D tomographic image reconstruction problem \cite{Natterer2}.  The phantom test image is shown in Figure \ref{fig3} (a).  The data generated for tomography are 1-D line integrals of the image, well-known as Radon transform data.  Formally, the Radon transform of a 2-D function or image $f$ is defined as
\begin{equation}\label{Radon}
 R f (t,\theta) = \int_\Omega f(x,y) \delta(t-(x,y)\cdot(\cos \theta ,\sin \theta)) \, dx \, dy,
\end{equation}
where $\Omega$ is the image domain and $\delta$ is the Dirac delta function.  As in many applications, the data collected for reconstruction are of the form known as parallel beam geometry.  In this setting, the full knowledge of noisy $Rf(t,\theta)$ is known for some finite set of angles, $\theta.$\footnote{There is also a discretization over $t$, but it is small enough to ignore.}  In this numerical experiment, we use a total of 29 angles that are equally distributed across the full 180$\degree$ angular range, which are visualized as a sinogram in Figure \ref{fig3} (b).  Such a limited set of data is sometimes referred to as \emph{limited data} tomography.  Mean zero normally distributed noise was again added to the data values.  Classically tomographic reconstruction from parallel beam geometry can be done by first transforming the data into Fourier space by the Fourier slice theorem, and then applying a chosen ramp filter to this data and the inverse Fourier transform.  This direct approach, called filtered backprojection, is sensitive to noise and is shown in Figure \ref{fig3} (c).

The problem can however be discretized and approximated by a set of linear equations $Af = b$ (see for instance \cite{sanders2015image} on pages 8-9 within section 1.5.), where $A$ is sparse matrix that is a discretized approximation of the Radon transform, $f$ is the vectorized image, and $b$ is a vector holding the data values.  With this set up we can apply regularization models such as (\ref{gen-l1}) and (\ref{MHOTV-model}).  We use a $512\times 512$ pixelated mesh for the image domain in this experiment.  The results for applying these models with HOTV, MHOTV, and Daubechies wavelets all at orders one and three are shown in Figure \ref{fig4}.  Each of the models are also supplemented with a nonnegativity constraint, $f\ge 0$, which is carried out with a projected gradient method.  A baseline comparison obtained by a conjugate gradient least squares solver is also shown in the figure.  To ensure accurate comparison between the methods by appropriate parameter selection and algorithmic convergence, the relative data errors defined by $\| Af - b \|_2 / \| b\|_2$ are shown in the figure, and it confirms that each approach approximately fit the data equally well, with all of the errors contained within an interval of size 0.0129.

As has been observed previously \cite{SGP-ET}, due to a number of reasons including undersampling, noise, fine details between the image features, and nature of the regularization, the order 1 solutions (TV) can leave the fine features under resolved, even though the underlying image is truly a piecewise constant that classical TV was originally designed to recover.  Each of these order 1 images appear relatively similar, with the MHOTV and Daubechies approaches showing modest improvements in resolving some features.  As in the 1-D case, the HOTV3 solution exhibits some small local oscillations that appear as noise in the image.  However, this image, as well as the other order 3 approaches resolve the features notably more clear than the order 1 approaches.  Both of the order 3 multiscale approaches appear less noisy than the HOTV order 3, while still maintaining a good approximation of the image features.

\begin{figure}
 \centering
 \includegraphics[width=0.5\textwidth]{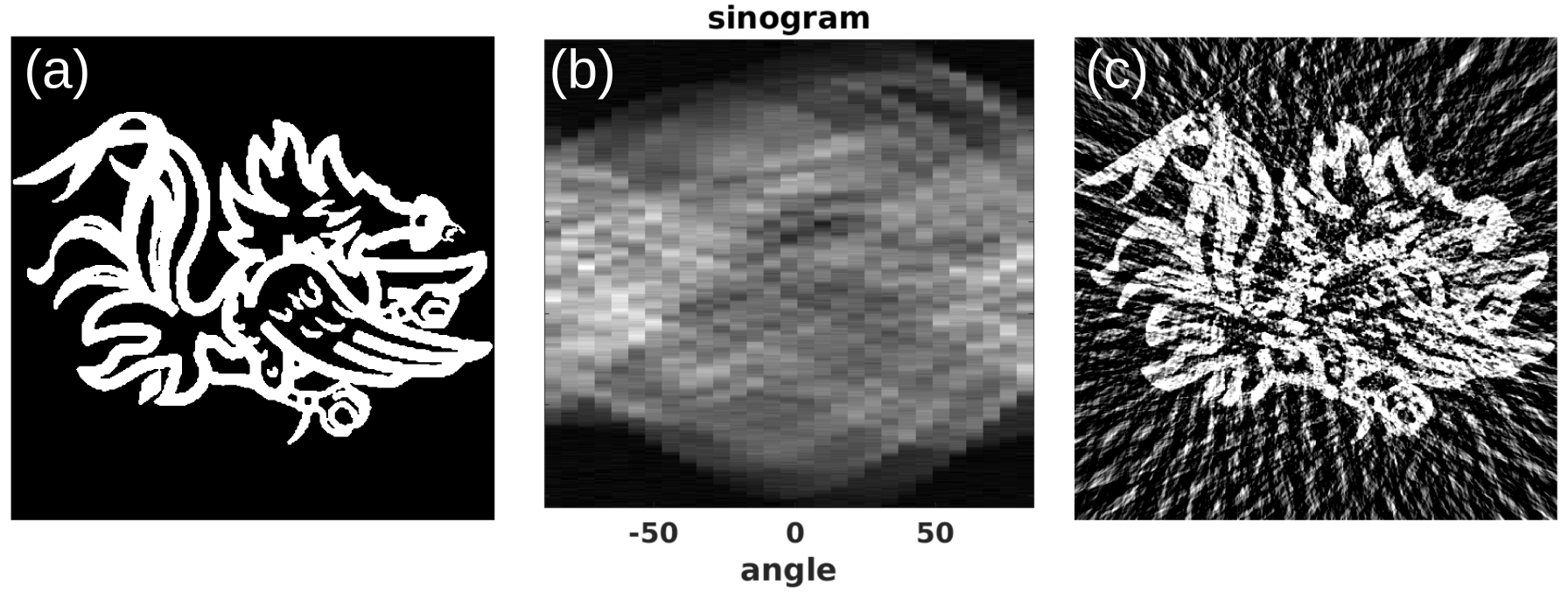}
 \caption{(a) Phantom image.  (b) Noisy tomographic data in sinogram format, 29 projections in total.  (c) Classical filtered backprojection reconstruction from data.}
 \label{fig3}
\end{figure}

\begin{figure*}
 \centering
 \includegraphics[scale=.8]{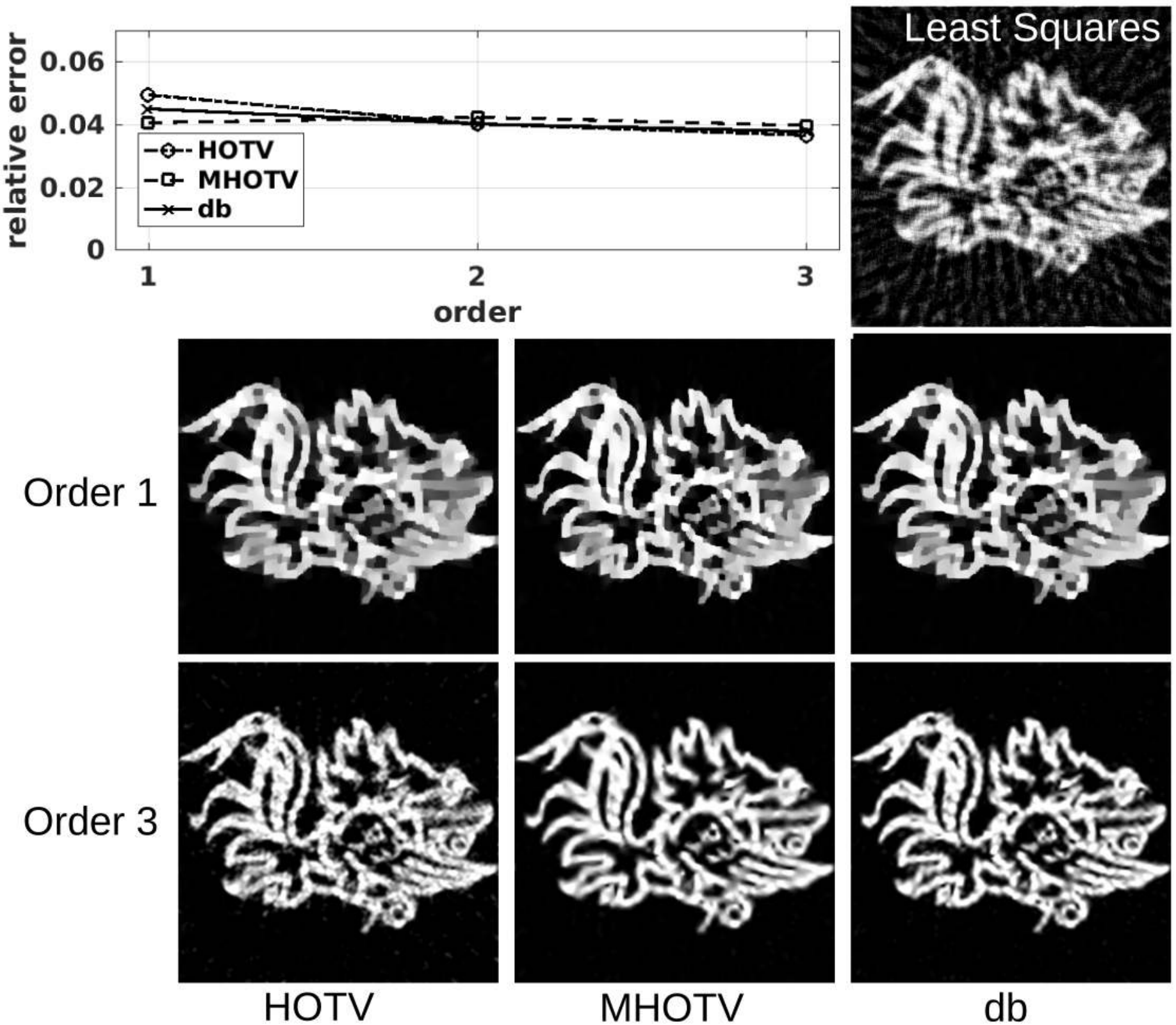}
 \caption{Reconstructions of phantom image from 29 tomographic projections.  Orders 1 and 3 are shown for the regularization approaches.  Top left: relative data fitting error from each approach shows approximately equivalent data fitting from each approach.  Top right: least squares solution for baseline comparison.}
 \label{fig4}
\end{figure*}

\begin{table*}
\begin{center}
\begin{tabular}{r | c | c | c | c | c | c |}
            SNR   & & Order 1 & Order 2 & Order 3 & Order 1.5 & Order 2.5 \\ \hline
      &          &mhotv \vline\,Daub&mhotv \vline \,Daub&mhotv \vline \,Daub&mhotv& mhotv\\ 
             \hline
      & 1 level  &  .1624, .1624  & .2039, .1961   & .2464, .2306   &  .1819   & .2328 \\
 2    & 2 levels &  .1612, .1617 & .1742, .1852   & .2135, .2223   &  .1782    & .2183 \\
      & 3 levels &  .1699, .1615  & \textbf{.1513}, .1778   & .1745, .2149   &  .1776    & .1975 \\
      & 4 levels &  .2001, .1647 & .1584, .1745   & .1764, .2104   &  .2031    & .2102 \\
      \hline
      & 1 level  &.0864, .0864& .0971, .0914 & .1293, .1090 & .1025&.1287 \\
 5    & 2 levels &.0858, .0857& .0761, .0838 & .0946, .1004 & .0987&.1172 \\
      & 3 levels &.0926, .0864& \textbf{.0668}, .0805 & .0766, .0982 & .1016&.1133 \\
      & 4 levels &.1100, .0894& .0742, .0801 & .0828, .0981 & .1186&.1276 \\
 \hline
       & 1 level  &.0543, .0542& .0509, .0480 & .0694, .0572 & .0690& .0841 \\
 10    & 2 levels &.0542, .0539& .0400, .0442 & .0489, .0528 & .0657& .0763 \\
       & 3 levels &.0589, .0547& \textbf{.0359}, .0430 & .0399, .0522 & .0694& .0776 \\
       & 4 levels &.0696, .0570& .0413, .0436 & .0442, .0535 & .0802& .0880 \\

       \hline
\end{tabular}
\end{center}
\caption{Average relative reconstruction error over 100 simulations, as a function of the order of the method and number of levels in the multiscale approaches.  The minimums for each SNR are emphasized in bold.}
\label{table1}
\end{table*}

\subsection{Quantitative Results}

We performed two sets of simulations to compare the methods in a more quantitative manner.  The first set of results presented here involved setting up 100 different test problems and then running all of our methods over each time for multiple noise levels, and the mean reconstruction error over all simulations is presented in Table \ref{table1}, with the MHOTV resulting in the left of each column and Daubechies wavelets in the right of each column.  It is important to note here, that the parameter $\lambda$ in \ref{MHOTV-model} was optimized in every reconstruction to yield the solution that minimized the true error between the test signal and the reconstruction, making for objective comparisons.  In order to set up each test problem, a 1D piecewise quadratic polynomial (presumably ideal for order 3) was randomly generated over a 1024 stencil, and the entries in sampling matrix $A \in \R^{1024\times 1024}$ and added noise to $b$ were randomly generated from a mean zero Gaussian distribution.  Overall, these results show that MHOTV moderately outperforms Daubechies wavelets in each case, and remaining comparisons between the order and number of levels are generally consistent between MHOTV and the wavelets.

For the single level case (original TV and HOTV), the error generally increases for higher orders, contrary to the results in previous work \cite{Archibald2015}.  Multiple scales show notable improvement for the higher orders, whereas they show a mild reduction in accuracy for order 1.  The most benefit for both orders 2 and 3 is seen when using 3 levels, and order 2 actually outperforms order 3.  Finally, using the fact that (\ref{eq: DFTeq}) gives us a way to compute fractional orders of the method, we present also the results from orders 1.5 and 2.5.  These are notably worse than the integer orders, a testament to the fact that these fractional order derivates result in highly nonlocal differences \footnote{To observe these nonlocal stencils, one can compute the inverse Fourier transform of (\ref{eq: DFTeq}) for fractional orders $k$.}.

In the second set of results presented here we ran a series of numerical simulations and measured the rate of {successful} recovery for each method as a function of the sampling rate. For each simulation we randomly generated a piecewise polynomial of specified maximal degree over a 1024 stencil.  This function was randomly sampled at the specified sampling rate precisely as in the previous 1-D simulations in section \ref{sec: repeat1d}, where the sampling rate is defined by the number of samples divided by the number of grid points.  Each regularization {procedure} was then used for reconstruction, and the $\ell_2$ error between the true function and reconstructed functions is determined.  If the error was less than $10^{-2}$, then the reconstruction was said to yield a \emph{successful recovery.}  This simulation was carried out for each sampling rate in 20 trials, and the fraction of those 20 trials that yielded {successful} recovery is set as our probability of success.  {In each case, the generated test functions had five discontinuities, and the location of the jumps were drawn randomly from a uniform distribution on the approximation interval.}

No noise was added for these simulations, as this can make the likelihood of an exact recovery unlikely.  Therefore, for this case our general $\ell_1$ model as a modification of (\ref{gen-l1}) is given by
\begin{equation}\label{exact-model}
f_{rec}  = \arg \min_f \| T f \|_1 \quad \text{s.t.} \quad Af = b,
\end{equation}
and similarly for our specific MHOTV model in (\ref{MHOTV-model}).
This constrained data fitting problem is solved by reformulating as an unconstrained problem with an augmented Lagrangian function \cite{hestenes1969multiplier, Li2013}.

The results for these simulations are shown in Figure \ref{fig5}.  The results are separated in two ways, by the degree of the piecewise polynomial function that is sampled (varying along the rows) and the order of the regularization method (varying along the columns).  In the first row are results for piecewise constant functions, in the second row are piecewise linear functions, and in the third row are piecewise quadratic functions.  In all cases, the MHOTV yields the highest probability of success, regardless of the degree of the polynomial or order of the regularization, and the Daubechies wavelets success appears to generally lie somewhere between MHOTV and HOTV.  The order 1 regularizations perform well only in the case of piecewise constant functions.  On the other hand, the order 2 and 3 regularizations perform well for all function types, with order 2 again outperforming order 3 both with piecewise linear and quadratic signals.

\begin{figure*}
 \centering
 \includegraphics[width =\textwidth]{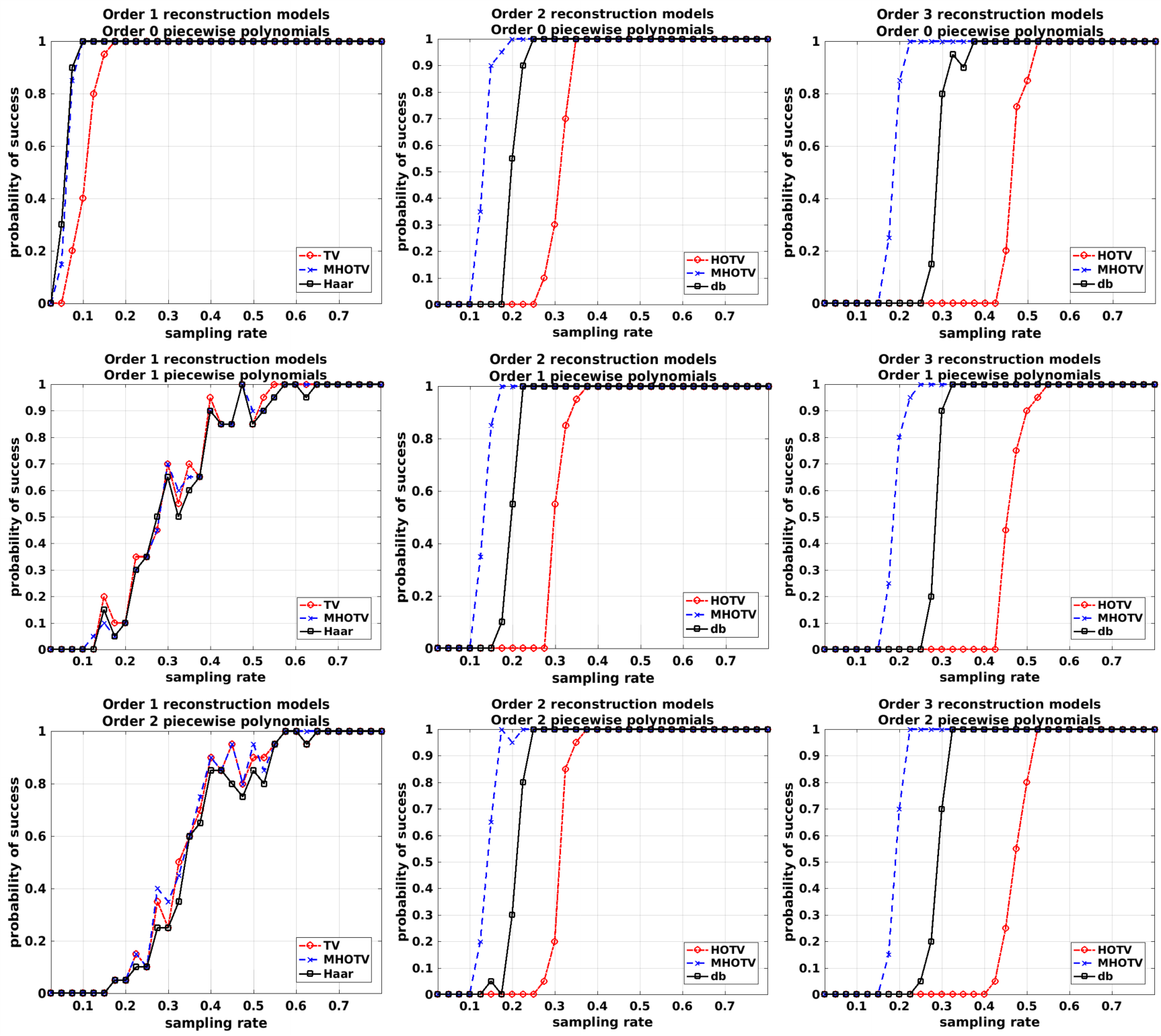}
 \caption{Probability of success for HOTV, MHOTV, and Daubechies wavelets at orders 1 (left column), 2 (middle column) and 3 (right column).  A successful recovery is deemed whenever the relative $\ell_2$ error between the reconstruction and the true signal is less than $10^{-2}$. Top row: piecewise constant functions.  Middle row: piecewise linear functions.  Bottom row: piecewise quadratic functions.}
 \label{fig5}
\end{figure*}

\section{Summary}

HOTV circumvents the staircasing often observed in TV solutions and has been shown to be more effective for problems with fine features, where resolution can be improved by increasing the order of derivatives in the regularization term \cite{SGP-ET}. In some applications, however, high order derivatives promote solutions with spurious local oscillations, as shown in Figure~\ref{fig1}. The MHOTV regularization we introduce in this work is shown to mitigate unwanted oscillations while maintaining the resolution power of high order regularization. 

{Although the theory for MHOTV reconstructions remains underdeveloped when compared to wavelets regularization \cite{eck1995multiresolution,tenoudji2016wavelets,guo2007optimally,kutyniok2012shearlets,starck2010sparse,gao1998wavelet,taswell2000and}, our experiments indicate that MHOTV can outperform wavelets regularization in practical applications.  Figure~\ref{fig2}, for instance, shows fewer spurious oscillations in the MHOTV reconstruction than for  Daubechies wavelets penalization. A feature that can also be observed for the 2-D tomographic data. Moreover, our results  show that MHOTV regularization requires fewer samples for successful reconstructions than for HOTV and wavelets. Computational efficiency is achieved by performing the transformation in Fourier space or by matrix decomposition, as derived in section \ref{sec5}. The associated matlab algorithms can be downloaded at \cite{toby-web}, and some of the simulations in the proceeding section can also be found there.}

\appendix
\section{Proof of Theorem \ref{thm1}}
\begin{lemma}\label{lem1}
Let $k,\ell \in \Z$ with $0\le \ell \le k$.  Then we have the following Vandermonde-like identity:
\begin{equation}\label{Vandermonde}
(-1)^p {k\choose p} = \sum_{j=0}^\ell (-1)^j {k\choose j} {k+1 \choose \ell - j},
\end{equation}
where $p = \ell/2$ for $\ell$ even and $p = (\ell - 1)/2$ for $\ell$ odd.
\end{lemma}
\begin{proof}[Proof of lemma \ref{lem1}]
Consider the polynomial $p(x) = (1-x^2)^k (1+x) $, which can be factored as $p(x) = (1-x)^k (1+x)^{k+1}$.  Both representations can be expanded using the binomial sum giving
\begin{equation}\label{vande1}
p(x) = \sum_{j=0}^k (-x^2)^j {k \choose j} (1+x) = \sum_{j=0}^k (-1)^j {k\choose j}\left[x^{2j} + x^{2j+1} \right]
\end{equation}
by the first representation and 
\begin{equation}\label{vande2}
p(x) = \left[\sum_{j=0}^k (-x)^j {k\choose j} \right] \left[ \sum_{j=0}^k x^j {k+1 \choose j} \right]
\end{equation}
by the second representation.  Since (\ref{vande1}) and (\ref{vande2}) are equivalent for all $x$, the coefficients of any particular power of $x$ are equivalent, which is the equality we set out to prove.
\end{proof}

\begin{proof}[Proof of theorem \ref{thm1}]
Statement 3 is an immediate consequence of statement 2, since each matrix involved in the product is a convolution operator, and convolution operations are commutative and associative.  

To prove statement 1, first observe that with increasing $m$, the nonzero entries in the rows of $P_m$ become increasingly spaced, and it easy to see that the general resulting product $P_m^{k+1}$ is essentially the same for each $m$ with different spacings between the nonzero entries.  Thus it is enough to show statement 1 for $m=1$.  In the case $k=1$, this calculation can be checked directly.  So suppose 1 holds for some arbitrary $k$.  Then we need to show that $(P_1 P_1^{k+1})$ yields the desired result as defined by (\ref{pvalues}).  It is fairly easy to see that the resulting entries of this product is simply the addition of two neighboring entries (modulo $N$) in $P_1^{k+1}$.  Any such entries added together trivially yields the desired values, and the proper location of these values is also easy to confirm.

Similar arguments used for statement 1 also apply to statement 2.  First, we can consider an inductive approach, over $m$, where we will need to show $\Phi_{k,2^{m+1}} = P_{m+1}^{k+1} \Phi_{k,2^{m}}$.  Note that again due to the spacing of the entries of $P_m^{k+1}$, the argument for any arbitrary $m$ is parallel to that for $m=1$, with just different handling of the indices.  Therefore the case for $m=1$ suffices for the inductive step, and the case for $m=1$ is an immediate consequence of the previous lemma.
\end{proof}

\section{Definitions}
If $f,g\in \R^N$, then the convolution of $f$ and $g$ is given by
\begin{equation}\label{defconv}
 (f*g)_m = \sum_{n=0}^{N-1} f_n\, g_{m-n}, \quad \text{for} ~ m = 0,1,\dots,N-1,
\end{equation}
where for indices of $g$ running outside of domain of $g$, a periodic extension of $g$ is assumed.  The discrete Fourier transform (DFT) of $f$ is defined by
\begin{equation}\label{defDFT}
 \F (f)_\xi = \sum_{n=0}^{N-1} f_n e^{\frac{-i2\pi}{N}n\xi} \quad \text{for}~ \xi = 0,1,\dots,N-1,
\end{equation}
and the inverse discrete Fourier transform (IDFT) of $f$ is given by \
\begin{equation}\label{defIDFT}
 \F^{-1} (f)_n = \frac{1}{N} \sum_{\xi=0}^{N-1} f_\xi e^{\frac{i2\pi}{N}\xi n} \quad \text{for}~ n = 0,1,\dots,N-1.
\end{equation}


\section*{Acknowledgements}
This work is supported in part by the grants NSF-DMS 1502640, NSF-DMS 1522639 and AFOSR FA9550-15-1-0152.


\begin{thebibliography}{10}

\bibitem{Archibald2015}
R.~Archibald, A.~Gelb, and R.~B. Platte.
\newblock Image reconstruction from undersampled {F}ourier data using the
  polynomial annihilation transform.
\newblock {\em J. Sci. Comput.}, pages 1--21, 2015.

\bibitem{bhattacharya2007fast}
S.~Bhattacharya, T.~Blumensath, B.~Mulgrew, and M.~Davies.
\newblock Fast encoding of synthetic aperture radar raw data using compressed
  sensing.
\newblock In {\em 2007 IEEE/SP 14th Workshop on Statistical Signal Processing},
  pages 448--452. IEEE, 2007.

\bibitem{blomgren1997total}
P.~Blomgren, T.~F. Chan, P.~Mulet, C.-K. Wong, et~al.
\newblock Total variation image restoration: numerical methods and extensions.
\newblock In {\em ICIP (3)}, pages 384--387, 1997.

\bibitem{TGV}
K.~Bredies, K.~Kunisch, and T.~Pock.
\newblock Total generalized variation.
\newblock {\em SIAM J. Imaging Sci.}, 3(3):492--526, 2010.

\bibitem{CSincoherence}
E.~Cand\`{e}s and J.~Romberg.
\newblock Sparsity and incoherence in compressive sampling.
\newblock {\em Inverse Probl.}, 23(3):969, 2007.

\bibitem{candes2006robust}
E.~J. Cand{\`e}s, J.~Romberg, and T.~Tao.
\newblock Robust uncertainty principles: Exact signal reconstruction from
  highly incomplete frequency information.
\newblock {\em IEEE Transactions on information theory}, 52(2):489--509, 2006.

\bibitem{chambolle1997image}
A.~Chambolle and P.-L. Lions.
\newblock Image recovery via total variation minimization and related problems.
\newblock {\em Numerische Mathematik}, 76(2):167--188, 1997.

\bibitem{HOTV}
T.~Chan, A.~Marquina, and P.~Mulet.
\newblock High-order total variation-based image restoration.
\newblock {\em SIAM J. Sci. Comput.}, 22(2):503--516, 2000.

\bibitem{spinning}
R.~R. Coifman and D.~L. Donoho.
\newblock {\em Translation-invariant de-noising}.
\newblock Springer, 1995.

\bibitem{daubechies1992ten}
I.~Daubechies et~al.
\newblock {\em Ten lectures on wavelets}, volume~61.
\newblock SIAM, 1992.

\bibitem{eck1995multiresolution}
M.~Eck, T.~DeRose, T.~Duchamp, H.~Hoppe, M.~Lounsbery, and W.~Stuetzle.
\newblock Multiresolution analysis of arbitrary meshes.
\newblock In {\em Proceedings of the 22nd annual conference on Computer
  graphics and interactive techniques}, pages 173--182. ACM, 1995.

\bibitem{eldar2012compressed}
Y.~C. Eldar and G.~Kutyniok.
\newblock {\em Compressed sensing: theory and applications}.
\newblock Cambridge University Press, 2012.

\bibitem{gao1998wavelet}
H.-Y. Gao.
\newblock Wavelet shrinkage denoising using the non-negative garrote.
\newblock {\em Journal of Computational and Graphical Statistics},
  7(4):469--488, 1998.

\bibitem{bregman}
T.~Goldstein and S.~Osher.
\newblock The split {B}regman method for l1-regularized problems.
\newblock {\em SIAM J. Imaging Sci.}, 2(2):323--343, 2009.

\bibitem{guo2007optimally}
K.~Guo and D.~Labate.
\newblock Optimally sparse multidimensional representation using shearlets.
\newblock {\em SIAM journal on mathematical analysis}, 39(1):298--318, 2007.

\bibitem{hestenes1969multiplier}
M.~R. Hestenes.
\newblock Multiplier and gradient methods.
\newblock {\em Journal of optimization theory and applications}, 4(5):303--320,
  1969.

\bibitem{hu2012higher}
Y.~Hu and M.~Jacob.
\newblock Higher degree total variation {(HDTV)} regularization for image
  recovery.
\newblock {\em IEEE Transactions on Image Processing}, 21(5):2559--2571, 2012.

\bibitem{kamilov2012wavelet}
U.~Kamilov, E.~Bostan, and M.~Unser.
\newblock Wavelet shrinkage with consistent cycle spinning generalizes total
  variation denoising.
\newblock {\em IEEE Signal Processing Letters}, 19(4):187--190, 2012.

\bibitem{king2013image}
E.~J. King, G.~Kutyniok, and W.-Q. Lim.
\newblock Image inpainting: theoretical analysis and comparison of algorithms.
\newblock {\em SPIE Optical Engineering+ Applications}, page 885802, 2013.

\bibitem{kutyniok2012shearlets}
G.~Kutyniok et~al.
\newblock {\em Shearlets: Multiscale analysis for multivariate data}.
\newblock Springer Science \& Business Media, 2012.

\bibitem{Leary}
R.~Leary, Z.~Saghi, P.~A. Midgley, and D.~J. Holland.
\newblock Compressed sensing electron tomography.
\newblock {\em Ultramicroscopy}, 131:70 -- 91, 2013.

\bibitem{Li2013}
C.~Li, W.~Yin, H.~Jiang, and Y.~Zhang.
\newblock An efficient augmented lagrangian method with applications to total
  variation minimization.
\newblock {\em Comput. Optim. Appl.}, 56(3):507--530, 2013.

\bibitem{lustig2007sparse}
M.~Lustig, D.~Donoho, and J.~M. Pauly.
\newblock Sparse {MRI}: The application of compressed sensing for rapid {MR}
  imaging.
\newblock {\em Magnetic resonance in medicine}, 58(6):1182--1195, 2007.

\bibitem{1257394}
M.~Lysaker, A.~Lundervold, and X.-C. Tai.
\newblock Noise removal using fourth-order partial differential equation with
  applications to medical magnetic resonance images in space and time.
\newblock {\em IEEE Transactions on Image Processing}, 12(12):1579--1590, Dec
  2003.

\bibitem{4587391}
S.~Ma, W.~Yin, Y.~Zhang, and A.~Chakraborty.
\newblock An efficient algorithm for compressed {MR} imaging using total
  variation and wavelets.
\newblock In {\em Computer Vision and Pattern Recognition, 2008. CVPR 2008.
  IEEE Conference on}, pages 1--8, June 2008.

\bibitem{mallat2008wavelet}
S.~Mallat.
\newblock {\em A wavelet tour of signal processing: the sparse way}, pages
  292--296.
\newblock Academic press, 2008.

\bibitem{Natterer2}
F.~Natterer.
\newblock {\em The Mathematics of Computerized Tomography}.
\newblock Society for Industrial and Applied Mathematics, 2001.

\bibitem{ROF}
L.~I. Rudin, S.~Osher, and E.~Fatemi.
\newblock Nonlinear total variation based noise removal algorithms.
\newblock {\em Physica D: Nonlinear Phenomena}, 60(1):259--268, 1992.

\bibitem{toby-web}
T.~Sanders.
\newblock Matlab imaging algorithms: Image reconstruction, restoration, and
  alignment, with a focus in tomography.
\newblock \url{http://www.toby-sanders.com/software},
  \url{https://doi.org/10.13140/RG.2.2.33492.60801}.
\newblock Accessed: 2016-19-08.

\bibitem{sanders2015image}
T.~Sanders.
\newblock {\em Image Processing and 3-D Reconstruction in Tomography}.
\newblock PhD thesis, University of South Carolina, 2015.
\newblock Chapter 1: Background.

\bibitem{sanders2018parameter}
T.~Sanders.
\newblock Parameter selection for {HOTV} regularization.
\newblock {\em Applied Numerical Mathematics}, 125:1--9, 2018.

\bibitem{sanders2017subsampling}
T.~Sanders and C.~Dwyer.
\newblock Subsampling and inpainting approaches for electron tomography.
\newblock {\em Ultramicroscopy}, 182:292--302, 2017.

\bibitem{SGP-ET}
T.~Sanders, A.~Gelb, R.~Platte, I.~Arslan, and K.~Landskron.
\newblock Recovering fine details from under-resolved electron tomography data
  using higher order total variation regularization.
\newblock {\em Ultramicroscopy}, 174:97--105, 2017.

\bibitem{scherzer2000relations}
O.~Scherzer and J.~Weickert.
\newblock Relations between regularization and diffusion filtering.
\newblock {\em Journal of Mathematical Imaging and Vision}, 12(1):43--63, 2000.

\bibitem{setzer2011infimal}
S.~Setzer, G.~Steidl, and T.~Teuber.
\newblock Infimal convolution regularizations with discrete L1-type
  functionals.
\newblock {\em Communications in Mathematical Sciences}, 9(3):797--827, 2011.

\bibitem{starck2010sparse}
J.-L. Starck, F.~Murtagh, and J.~M. Fadili.
\newblock {\em Sparse image and signal processing: wavelets, curvelets,
  morphological diversity}.
\newblock Cambridge university press, 2010.

\bibitem{VOTV}
W.~Stefan, R.~A. Renaut, and A.~Gelb.
\newblock Improved total variation-type regularization using higher order edge
  detectors.
\newblock {\em SIAM J. Imaging Sci.}, 3(2):232--251, 2010.

\bibitem{steidl2006splines}
G.~Steidl, S.~Didas, and J.~Neumann.
\newblock Splines in higher order {TV} regularization.
\newblock {\em International journal of computer vision}, 70(3):241--255, 2006.

\bibitem{steidl2002relations}
G.~Steidl and J.~Weickert.
\newblock Relations between soft wavelet shrinkage and total variation
  denoising.
\newblock In {\em Joint Pattern Recognition Symposium}, pages 198--205.
  Springer, 2002.

\bibitem{steidl2004equivalence}
G.~Steidl, J.~Weickert, T.~Brox, P.~Mr{\'a}zek, and M.~Welk.
\newblock On the equivalence of soft wavelet shrinkage, total variation
  diffusion, total variation regularization, and sides.
\newblock {\em SIAM Journal on Numerical Analysis}, 42(2):686--713, 2004.

\bibitem{taswell2000and}
C.~Taswell.
\newblock The what, how, and why of wavelet shrinkage denoising.
\newblock {\em Computing in science \& engineering}, 2(3):12--19, 2000.

\bibitem{temizel2005wavelet}
A.~Temizel, T.~Vlachos, and W.~Visioprime.
\newblock Wavelet domain image resolution enhancement using cycle-spinning.
\newblock {\em Electronics Letters}, 41(3):119--121, 2005.

\bibitem{tenoudji2016wavelets}
F.~C. Tenoudji.
\newblock Wavelets; multiresolution analysis.
\newblock In {\em Analog and Digital Signal Analysis}, pages 337--373.
  Springer, 2016.

\bibitem{unser2017splines}
M.~Unser, J.~Fageot, and J.~P. Ward.
\newblock Splines are universal solutions of linear inverse problems with
  generalized {TV} regularization.
\newblock {\em SIAM Review}, 59(4):769--793, 2017.

\bibitem{Zhang}
Y.~Wang, J.~Yang, W.~Yin, and Y.~Zhang.
\newblock A new alternating minimization algorithm for total variation image
  reconstruction.
\newblock {\em SIAM J. Imaging Sci.}, 1(3):248--272, 2008.

\bibitem{wei2010sparse}
S.-J. Wei, X.-L. Zhang, J.~Shi, and G.~Xiang.
\newblock Sparse reconstruction for {SAR} imaging based on compressed sensing.
\newblock {\em Progress In Electromagnetics Research}, 109:63--81, 2010.

\end{thebibliography}
\end{document}